\documentclass[11pt]{amsart}
\usepackage{amsmath}
\usepackage{graphicx}
\usepackage{epstopdf}
\usepackage{amssymb, graphicx, amsmath, amsthm}
\usepackage{tikz}
\usetikzlibrary{arrows,decorations.markings,decorations.pathreplacing,calc,matrix,intersections}
 \usepackage{relsize}

\newtheorem*{propositionA}{Proposition~\ref{free-by-(torsion-free nilpotent)}}

\newtheorem{theorem}{Theorem}[section]
\newtheorem*{theorem*}{Theorem}
\newtheorem*{lemma*}{Lemma}
\newtheorem{proposition}[theorem]{Proposition}
\newtheorem{lemma}[theorem]{Lemma}
\newtheorem{corollary}[theorem]{Corollary}

\theoremstyle{definition}
\newtheorem{definition}[theorem]{Definition}
\newtheorem{example}[theorem]{Example}

\theoremstyle{remark}
\newtheorem{remark}[theorem]{Remark}

\usepackage{hyperref}
 
\begin{document}
\title{On subdirect products of type $FP_n$ of limit groups over Droms RAAGs}
\author{Dessislava H. Kochloukova, Jone Lopez de Gamiz Zearra}

\maketitle

\begin{abstract}
We generalize some known results for limit groups over free groups and residually free groups to limit groups over Droms RAAGs and residually Droms RAAGs, respectively.
We show that limit groups over Droms RAAGs are free-by-(torsion-free nilpotent). We prove that if $S$ is a full subdirect product of type $FP_s(\mathbb{Q})$ of limit groups over Droms RAAGs with trivial center, then the projection of $S$ to the direct product of any $s$ of the limit groups over Droms RAAGs has finite index. 
Moreover, we compute the growth of homology groups and the volume gradients for limit groups over Droms RAAGs in any dimension and for finitely presented residually Droms RAAGs of type $FP_m$ in dimensions up to $m$. In particular, this gives the values of the analytic $L^2$-Betti numbers of these groups in the respective dimensions.
\end{abstract}

\section{Introduction}

The class of \emph{right-angled Artin groups (RAAGs)} generalizes the class of finitely generated free groups, by allowing relations saying that some of the generators commute. In recent years, RAAGs have been extensively studied due to their rich structure, in particular, due to the variety of their subgroups. For example, it follows from works of I. Agol (\cite{Agol}) and  D. Wise (\cite{Wise}) that fundamental groups of closed, irreducible, hyperbolic $3$-manifolds are virtually subgroups of RAAGs.

In general, not all subgroups of RAAGs are again RAAGs, and Droms RAAGs are precisely those RAAGs with the property that all of their finitely generated subgroups are again RAAGs. In particular, finitely generated free groups are Droms RAAGs. Droms RAAGs can also be described as the ones where the defining graph does not contain induced squares or straight line paths with $3$ edges (see \cite{Droms2}). Alternatively, the class of Droms RAAGs can be described as the $Z \ast$-closure of $\mathbb{Z}$ (see Section \ref{Section 2} for the definition).

The class of \emph{limit groups over Droms RAAGs} extends the class of limit groups over free groups.  The theory of limit groups over free groups was developed by O. Kharlampovich, A. Miasnikov, under the name of \emph{fully residually free groups}, and by Z. Sela, as a method for obtaining information on solutions of systems of equations over free groups. The class of limit groups over free groups contains all finitely generated free abelian groups and surface groups of Euler characteristic at most $-2$. Limit groups over free groups are commutative transitive and as shown by F. Dahmani (\cite{Dahmani}) they satisfy the Howson property.

The study of limit groups over Droms RAAGs was initiated by M. Casals-Ruiz, A. Duncan and  I. Kazachkov, in \cite{Montse} and \cite{Montse2}, with the aim of studying systems of equations over RAAGs. A group $\Gamma$ is a limit group over a Droms RAAG if it is precisely a finitely generated group that is fully residually a Droms RAAG, that is, for every finite subset $S$ of $\Gamma$ there is a Droms RAAG $G$ together with a homomorphism of groups $\varphi \colon \Gamma \mapsto G$ whose restriction to $S$ is injective.

Limit groups over Droms RAAGs share many properties with limit groups over free groups, including homological and homotopical finiteness properties. For example, M. Casals-Ruiz, A. Duncan and  I. Kazachkov showed in \cite{Montse} that limit groups over coherent RAAGs are finitely presented and they are subgroups of ICE groups over coherent RAAGs (see definition \ref{defIce}). In particular, this holds for limit groups over Droms RAAGs and implies that limit groups over Droms RAAGs are of type $FP_{\infty}$, and that their finitely generated subgroups are again limit groups over Droms RAAGs. In \cite{Bridson3} M. Bridson and H. Wilton proved that finitely presentable subgroups of  finitely generated residually free groups are separable and that the subgroups of type $FP_{\infty}$ are virtual retracts.  Though limit groups over Droms RAAGs share many properties with limit groups over free groups, there are properties like commutative transitivity and the Howson property that do not hold for a general Droms RAAG. For example, by \cite{Delgado}, the RAAGs that satisfy the Howson property are precisely the free products of abelian groups, so in particular, are limit groups over free groups.

In \cite{Desi} D. Kochloukova showed that limit groups over free groups are free-by-(torsion-free nilpotent). Our first result is a generalization of this fact for limit groups over Droms RAAGs. The rest of the results in this paper will be heavily dependent on this fact.

\begin{propositionA}
Every limit group over a Droms RAAG is free-by-(torsion-free nilpotent).
\end{propositionA}

Let us recall some definitions. Let $S$ be a subgroup of a direct product $G_1\times \cdots \times G_n$. We say that $S$ is \emph{full} if $S \cap G_i \neq 1$ for all $i\in \{1,\dots,n\}$ and $S$ is a \emph{subdirect product} if $p_i(S)= G_i$ for all $i\in \{1,\dots,n\}$, where $p_i$ is the projection map $G_1\times \cdots \times G_n \mapsto G_i$.

G. Baumslag, A. Miasnikov, V. Remeslennikov proved in \cite{Baumslag} that every finitely generated residually Droms RAAG is a subgroup of a direct product of finitely many limit groups over Droms RAAGs. 
The study of subdirect products of free groups was initiated by G. Baumslag and J. Roseblade in \cite{Roseblade}. In a sequence of papers that culminated in \cite{Bridson2} and \cite{Bridson}, M. Bridson, J. Howie, C. Miller and H. Short studied subgroups of the direct product of limit groups over free groups. One of the results is that if $\Gamma_1,\dots, \Gamma_n$ are non-abelian limit groups over free groups and $S$ is a full subdirect product of $\Gamma_1\times \cdots \times \Gamma_n$, then $S$ is of type $FP_2(\mathbb{Q})$ if and only if $p_{i,j}(S)$ is of finite index in $\Gamma_i \times \Gamma_j$ for all $1\leq i< j \leq n$, where  $p_{i,j}$ denotes the projection map $S\mapsto \Gamma_i \times \Gamma_j$. These results were generalized by J. Lopez de Gamiz Zearra in \cite{Jone} to the class of limit groups over Droms RAAGs (see Section \ref{Limit groups}).
In addition, in \cite{Desi} D. Kochloukova generalized the previous result concerning limit groups over free groups by showing that if $S$ is of type $FP_s$, then it virtually surjects onto $s$ coordinates. However, the converse is still an open problem.

We recall that a group $S$ is of \emph{homological type $FP_s(R)$} (or it is $FP_s$ over $R$) for a commutative, associative ring $R$ with identity element if the trivial $R G$-module $R$ has a projective resolution with all modules finitely generated in dimensions up to $s$.  A group is of type $FP_s$ if it is $FP_s$ over $\mathbb{Z}$. The goal of Section \ref{Section 5} is to extend the above results to the class of limit groups over Droms RAAGs.

\begin{theorem}\label{FPs}
Let $G_1, \dots,   G_m$ be limit groups over Droms RAAGs such that each $G_i$ has trivial center and let \[S < G \colon = G_1 \times \cdots \times G_m\] be a finitely generated full subdirect product. Suppose further that, for some fixed number $s \in \{ 2, \dots, m \}$, for every subgroup $S_0$ of finite index in $S$ the homology group $H_i (S_0 , \mathbb{Q})$ is finite dimensional (over $\mathbb{Q}$) for all $ i \leq  s.$ Then for every canonical projection
\[p_{j_1, \dots,  j_s} \colon S \mapsto G_{j_1} \times \cdots \times G_{ j_ s}\]
the index of $p_{j_1, \dots, j_s} (S)$ in  $G_{j_1} \times \ldots \times G_{ j_s}$  is finite.

In particular, if $S$ is of type $FP_s$ over $\mathbb{Q}$, then for every canonical projection $p_{j_1, \dots,  j_s}$ the index of $p_{j_1, \dots, j_s} (S)$ in  $G_{j_1} \times \cdots \times G_{ j_s}$  is finite.    
\end{theorem}

In Section \ref{Section 6} we discuss the growth of homology groups and the volume gradients of limit groups over Droms RAAGs and of finitely presented residually Droms RAAGs.
The \emph{growth of homology groups} is measured by the limit \[\lim_{n \to \infty  } \dim_{K} H_{i}(B_{n}, K) \slash [G \colon B_{n}],\] whenever this limit exits. Here $K$ is a field and  $(B_{n})$ is an \emph{exhausting normal chain} of $G$, that is, a chain $(B_{n})$ of normal subgroups of finite index such that $B_{n+1} \subseteq B_{n}$ and $\bigcap_{n} B_{n}=1$. As every residually finite group has an exhausting normal chain, the same holds for limit groups over coherent RAAGs. Thanks to the \emph{approximation theorem} of W. L\"uck we know that if $K$ has characteristic $0$, $G$ is residually finite, finitely presented and of type $FP_{m}$, then the limit exists for $i<m$ and equals the $L^{2}$-Betti number $\beta_{i}(G)$. In particular, in this case the limit is independent of the exhausting normal chain.

Given a group $G$ of homotopical type $F_{m}$, $\text{vol}_{m}(G)$ was defined in \cite{Desi2} as the least number of $m$-cells among all classifying spaces $K(G,1)$ with finite $m$-skeleton. For instance, $\text{vol}_{1}(G)$ equals $d(G)$, where $d(G)$ is the minimal number of generators of $G$.
The most studied volume gradient is the $1$-dimensional one. That is, the \emph{rank gradient of $G$ with respect to $(B_{n})$}, which is defined to be $\lim_{n } d(B_{n}) \slash [G \colon B_{n}]$ and it is denoted by $RG(G,(B_{n}))$. Its study was initiated by M. Lackenby in \cite{Lack} and a combinatorial approach to the rank gradient was developed by M. Abert, N. Nikolov and A. Jaikin-Zapirain in \cite{ANZ}. In general, the \emph{$m$-dimensional volume gradient of $G$ with respect to $(B_{n})$} is $$\lim_{n \to \infty} \text{vol}_{m}(B_{n}) \slash [G \colon B_{n}].$$ In \cite{Desi2} M. Bridson and D. Kochloukova calculated the above volume and homology type gradients for limit groups over free groups and in the case of residually free groups they found particular filtrations where the homology growth and the rank gradient can be calculated. The results about homology growth in \cite{Desi2} and in this paper are independent of the characteristic of the field $K$.

In Section \ref{Section 6}  we compute the homology growth and the volume gradients for a more general setting, namely, for limit groups over coherent RAAGs. Note that this class of groups is more general than the class of limit groups over Droms RAAGs. We use the approach from \cite{Desi2} and we show that limit groups over coherent RAAGs are slow above dimension 1, hence are $K$-slow above dimension 1.  This  implies the following results.

\begin{theorem}\label{Theorem C}
Let $G$ be a limit group over a coherent RAAG and let $(B_{n})$ be an exhausting normal chain in $G$. Then,\\[5pt]
(1) the rank gradient $RG(G, (B_n)) = \lim_{n\to \infty} \frac{d(B_{n})}{[G \colon B_{n}]}= -\chi(G).$\\[5pt]
(2) the deficiency gradient $DG(G, (B_n)) = \lim_{n\to \infty} \frac{\text{def}(B_{n})}{[G \colon B_{n}]}= \chi(G)$.\\[5pt]
(3) the $k$-dimensional volume gradient $\lim_{n \to \infty} \frac{\text{vol}_{k}(B_{n})}{[G \colon B_{n}]} =0$ for $k\geq 2$.    
\end{theorem}

\begin{theorem}\label{Theorem D}
Let $K$ be a field, $G$ a limit group over a coherent RAAG and $(B_{n})$ an exhausting normal chain  in $G$. Then,\\[5pt]
(1) $\lim_{n\to \infty} \frac{\dim_{K}H_{1}(B_{n},K)}{[G \colon B_{n}]}=-\chi(G).$\\[5pt]
(2) $\lim_{n\to \infty} \frac{\dim_{K} H_{j}(B_{n} , K)}{[G \colon B_{n}]}=0$ for all $j\geq 2$.    
\end{theorem}

Using results from \cite{Desi2} we also compute the homology growth up to dimension $m$ and the volume gradients in low dimensions of residually Droms RAAGs of type $FP_m$ for $m \geq 2$. We cannot apply the same method to the class of residually coherent RAAGs since Theorem \ref{FPs} is a key result in this method and in \cite{Jone-Montse} it is proved that this no longer holds for coherent RAAGs.

\begin{theorem}\label{Theorem E}
Let $m\geq 2$ be an integer, let $G$ be a residually Droms RAAG of type $FP_m$, and let $\rho$ be the largest integer such that $G$ contains a direct product of $\rho$ non-abelian free groups. Then, there exists an exhausting sequence $(B_n)$ in $G$ so that for all fields $K$,\\[5pt]
(1) if $G$ is not of type $FP_{\infty}$, then \[\lim_{n  \to \infty } \frac{\dim H_i(B_n, K)}{[G \colon B_n]}=0 \hbox{ for all }0\leq i \leq m;\] 
(2) if $G$ is of type $FP_{\infty}$ then for all $j\geq 1$,
\[ \lim_{n   \to \infty } \frac{\dim H_{\rho} (B_n, K)}{[G \colon B_n]}= (-1)^{\rho} \chi(G), \quad \lim_{n \to \infty } \frac{\dim H_{j} (B_n, K)}{[G \colon B_n]}=0 \quad \text{for all} \quad j\neq \rho.\]    
\end{theorem}

\begin{remark}
The exhausting sequence $(B_n)$ will be constructed to be an exhausting normal chain in a subgroup of finite index in $G$.   
\end{remark}

   The following result is a low dimensional homotopical version of Theorem \ref{Theorem E}.

\begin{theorem}\label{Theorem F}
Every group $G$ that is a finitely presented residually Droms RAAG but is not a limit group over a Droms RAAG admits an exhausting normal chain $(B_n)$ in $G$ with respect to which the rank gradient \[ RG(G,(B_n))= \lim_{n\to \infty} \frac{d(B_n)}{[G \colon B_n]}=0.\]
Furthermore, if $G$ is of type $FP_3$ but it is not commensurable with a product of two limit groups over Droms RAAGs, $(B_n)$ can be chosen so that the deficiency gradient $DG(G,(B_n))=0.$    
\end{theorem}
\medskip

\textit{Acknowledgements.} We are thankful to Montserrat Casals-Ruiz for helpful talks during the preparation of this manuscript and the referee for the suggestions that improved the paper. D. Kochloukova was partially supported by the CNPq grant 301779/2017-1  and by the FAPESP grant 2018/23690-6. J. Lopez de Gamiz Zearra was supported by the Basque Government grant IT974-16 and the Spanish Government grant MTM2017-86802-P.

\section{Preliminaries}
\label{Section 2}

\subsection{Right-angled Artin groups} \label{def-def}

Let us recall the definition of a right-angled Artin group. Given a finite simplicial graph $X$ with vertex set $V(X)$, the corresponding \emph{right-angled Artin group (RAAG)}, denoted by $GX$, is given by the following presentation: 
\[ GX= \langle V(X) \mid xy=yx \iff x \text{ and } y \text{ are adjacent}\rangle.\]

Subgroups of RAAGs can be wild, in particular, not all subgroups of RAAGs are themselves RAAGs. RAAGs that have all their finitely generated subgroups again of this type are known as \emph{Droms RAAGs}.  Droms RAAGs can also be described as the ones where the defining graph does not contain induced squares or straight line paths with $3$ edges (see \cite{Droms2}).

 A group is \emph{coherent} if all its finitely generated subgroups are also finitely presented.  Droms proved that a RAAG is coherent if and only if its underlying
graph contains no full subgraph isomorphic to an $n$-cycle, for $n \geq 4$, see \cite{Droms}. In particular, Droms RAAGs are coherent RAAGs.

\begin{definition}\label{Definition hierarchy}
Let $\mathcal{G}$ be a class of groups. The \emph{$Z \ast$-closure of $\mathcal{G}$}, denoted by $Z \ast (\mathcal{G})$, is the union of classes $(Z \ast (\mathcal{G}))_{k}$ defined as follows. At level $0$, the class $(Z \ast (\mathcal{G}))_{0}$ equals $\mathcal{G}$. A group $G$ lies in $(Z \ast (\mathcal{G}))_{k}$ if and only if
\[ G \simeq \mathbb{Z}^m \times (G_{1}\ast \cdots \ast G_{n}),\] where $m\in \mathbb{N}\cup \{0\}$ and the group $G_{i}$ lies in $(Z \ast (\mathcal{G}))_{k-1}$ for all $i\in \{1,\dots,n\}$.

The \emph{level of $G$}, denoted by $l(G)$, is the smallest $k$ for which $G$ belongs to  $(Z \ast (\mathcal{G}))_{k}$.
\end{definition}

In this terminology, Droms showed in \cite{Droms2} that the class of Droms RAAGs is the $Z\ast-$closure of $\mathbb{Z}$. Analogously, it is the $Z\ast-$closure of the class of finitely generated free groups. We use the latter when we fix the level of a Droms RAAG.

\begin{example}
If $G$ is a Droms RAAG such that $l(G)=0$, then $G$ is a finite rank free group. If $l(G)=1$, then \[ G \simeq \mathbb{Z}^m \times F,\] for some $m \geq 1$, $F$ free of finite rank and $G$ is not $\mathbb{Z}$. If $l(G)=2$, then
\[G \simeq \mathbb{Z}^m \times \big{(}  (\mathbb{Z}^{n_{1}} \times F_{k_{1}}) \ast \cdots \ast (\mathbb{Z}^{n_{l}} \times F_{k_{l}})\big{)}, \] for some $m,n_{i}, k_{i} \in \mathbb{N} \cup \{0\}$, $\sum_i n_i \geq 1$, $l\geq 2$ and  for $i\in \{1,\dots,l\}$, $F_{k_i}$ free of finite rank $k_i$.
\end{example}

\subsection{Limit groups over Droms RAAGs}
\label{Limit groups}

 \begin{definition}  \label{def-1}
 A group $\Gamma$ is a limit group over a Droms RAAG if it is precisely a finitely generated group that is fully residually a Droms RAAG, that is, for every finite subset $S$ of $\Gamma$ there is a Droms RAAG $G$ together with a homomorphism of groups $\varphi \colon \Gamma \mapsto G$ whose restriction to $S$ is injective.
\end{definition} 

\begin{definition} \label{def-2}
If $G$ is a group, a {limit group over $G$} is a group $H$ that is finitely generated and fully residually $G$, that is, for any finite set of non-trivial elements $S \subseteq H$ there is a homomorphism $\varphi \colon H \mapsto G$ which is injective on $S$.
\end{definition} 

Note that in Definition \ref{def-2} $G$ is a fixed group but in Definition \ref{def-1} $G$ actually depends on $S$.

\begin{lemma}  A group $\Gamma$ is a limit group over a Droms RAAG if and only if there is a Droms RAAG $G$ such that $\Gamma$ is a limit group over $G$.
\end{lemma}

\begin{proof} It suffices to show that every group $\Gamma$ that satisfies Definition \ref{def-1} is  a limit group over $G$ for some appropriate Droms RAAG $G$. 
By definition, for any finite set of non-trivial elements $S \subseteq \Gamma$ there is a homomorphism $\varphi_S \colon \Gamma \mapsto G_S$ which is injective on $S$, where $G_S$ is a Droms RAAG. Then by substituting $G_S$ if necessary with $\text{im}(\varphi_S)$ (since $\text{im}(\varphi_S)$ is a Droms RAAG) we can assume that $d(G_S) \leq d(\Gamma)$.  But since there are only finitely many Droms RAAGs with at most $d(\Gamma)$ generators,  say $D_1, \dots, D_t$, we conclude that $G_S \in \{ D_1, \ldots, D_t \}$ and  $G_S$ is a subgroup of the fixed Droms RAAG $G \colon = D_1 \ast \cdots \ast D_t$.
\end{proof}

By  \cite[Theorem F1]{Baumslag}, groups that are residually Droms RAAGs are precisely finitely generated subgroups of a  finite direct product of limit groups over Droms RAAGs. In this section we state the basic properties of residually Droms RAAGs and limit groups over Droms RAAGs that we need in later sections.

 Classical limit groups, i.e. limit groups over free groups, admit  a hierarchical structure. Limit groups over free groups are the finitely generated subgroups of $\omega$-
residually free tower groups, known as $\omega$-rft groups (see \cite{K-M} and \cite{Sela}). An $\omega$-rft space of height $h$  is defined by induction on $h$ and an $\omega$-rft group is
the fundamental group of an $\omega$-rft space.
A height 0 space is the wedge of a finite collection of circles, closed
hyperbolic surfaces and tori of arbitrary dimension, excluding the closed surface of
Euler characteristic $-1$.
An $\omega$-rft space $Y_h$ of height $h$ is obtained from an $\omega$-rft space $Y_{h-1}$ of height $h-1$ by
attaching a block of abelian or quadratic type (for details see  \cite{B-H}, \cite{Sela}). The classical height of a limit group (over a free group) $S$ is the minimal height of an $\omega$-rft group that has a subgroup isomorphic to $S$.

 Limit groups over Droms RAAGs admit a hierarchical structure that comes from the fact that Droms RAAGs are the direct product of a free abelian group (possibly trivial) and a free product, so we can use the work of M. Casals-Ruiz and I. Kazachkov on limit groups over free products (see \cite{Montse3}) to obtain that hierarchy.

\begin{proposition}\cite[Proposition 2.1]{Bridson}\label{Graph tower 1}
Let $G$ be a Droms RAAG such that $l(G)=0$ (that is, a finitely generated free group) and let $\Gamma$ be a limit group over $G$ of height $h(\Gamma)$.

If $h(\Gamma)=0$, then $\Gamma$ equals $M_{1}\ast \cdots \ast M_{j}$ and $M_{1},\dots,M_{j}$ are free abelian groups or surface groups with Euler characteristic at most $-2$.

If $h(\Gamma)\geq 1$, then $\Gamma$ acts cocompactly on a tree $T$ where the edge stabilizers are trivial or infinite cyclic and the vertex groups are limit groups over $G$ of height at most $h(\Gamma)-1$.  Moreover, at least one of the vertex groups is nonabelian.
\end{proposition}

\begin{proposition}\cite[Theorem 8.2]{Montse3} \label{Height}\label{Graph tower 2}
Let $G$ be a Droms RAAG with $l(G) \geq 1$ so that $G$ is of the form $\mathbb{Z}^m \times (G_{1}\ast \cdots \ast G_{n})$ and $l(G_{i}) \leq l(G)-1$ for $i\in \{1,\dots,n\}$ and let $\Gamma$ be a limit group over $G$. Then, $\Gamma$ is $\mathbb{Z}^l \times \Lambda$ where $\Lambda$ is a limit group over $G_{1}\ast \cdots \ast G_{n}$ and if $m=0$, then $l=0$. If $h(\Lambda)=0$, then \[ \Lambda= A_{1}\ast \cdots \ast A_{j},\] where for each $t\in \{1,\dots,j\}$, $A_{t}$ is a limit group over $G_{i}$ for some $i\in \{1,\dots,n\}$.

If $h(\Lambda) \geq 1$, then $\Lambda$ acts cocompactly on a tree $T$ where the edge stabilizers are trivial or infinite cyclic and the vertex groups are limit groups over $G_{1}\ast \cdots \ast G_{n}$ of height at most $h(\Lambda)-1$.  Moreover, at least one of the vertex groups has trivial center.
\end{proposition}

We observe that the notion of  height in limit groups over Droms RAAGs given in Proposition \ref{Graph tower 2} is a natural generalization of the notion of height in limit groups over free groups given in 
Proposition \ref{Graph tower 1} except that, for historical reasons, the definition of limit groups over free groups of height $0$ includes more than free products of free abelian groups and allows the free factors to be surface groups with Euler characteristic at most $-2$.

In \cite{Jone}, following the paper \cite{Bridson}, the results concerning subgroups of direct products of limit groups over free groups were generalized to the case of limit groups over Droms RAAGs.

\begin{theorem}\cite[Theorem 3.1]{Jone}\label{Theorem 3.1}
If $\Gamma_{1},\dots, \Gamma_{n}$ are limit groups over Droms RAAGs and $S$ is a subgroup of $\Gamma_{1}\times \cdots \times \Gamma_{n}$ of type $FP_{n}(\mathbb{Q})$, then $S$ is virtually a direct product of limit groups over Droms RAAGs.
\end{theorem}

\begin{theorem}\cite[Theorem 8.1]{Jone}\label{Theorem 8.1}
Let $\Gamma_{1},\dots,\Gamma_{n}$ be limit groups over Droms RAAGs such that each $\Gamma_i$ has trivial center and let $S< \Gamma_{1}\times \cdots \times \Gamma_{n}$ be a finitely generated full subdirect product. Then either:\\[3pt]
(1) $S$ is of finite index; or\\[3pt]
(2) $S$ is of infinite index and has a finite index subgroup $S_{0}<S$ such that $H_{j}(S_{0}, \mathbb{Q})$ has infinite dimension for some $j\leq n$.
\end{theorem}

Another result that will be used later is associated to finitely presented residually Droms RAAGs. In order to state the result, we introduce the following definition: an embedding $S  \hookrightarrow \Gamma_{0}\times \cdots \times \Gamma_{n}$ of a finitely generated group $S$ that is residually a Droms RAAG as a full subdirect product of limit groups over Droms RAAGs is \emph{neat} if $\Gamma_{0}$ is abelian (possibly trivial), $S \cap \Gamma_{0}$ is of finite index in $\Gamma_{0}$ and $\Gamma_{i}$ has trivial center for $i\in \{1,\dots,n\}$.

\begin{theorem}\cite[Theorem 10.1]{Jone} \label{Theorem 10.1}
Let $S$ be a finitely generated group that is residually a Droms RAAG. The following are equivalent:\\[3pt]
(1) $S$ is finitely presentable;\\[3pt]
(2) $S$ is of type $FP_{2}(\mathbb{Q})$;\\[3pt]
(3) $\dim H_{2}(S_{0},\mathbb{Q})$ is finite for all subgroups $S_{0}<S$ of finite index;\\[3pt]
(4) there exists a neat embedding $S \hookrightarrow \Gamma_{0}\times \cdots \times \Gamma_{n}$ into a product of limit groups over Droms RAAGs such that the image of $S$ under the projection to $\Gamma_{i}\times \Gamma_{j}$ has finite index for $0\leq i<j \leq n$;\\[3pt]
(5) for every neat embedding $S \hookrightarrow \Gamma_{0}\times \cdots \times \Gamma_{n}$ into a product of limit groups over Droms RAAGs the image of $S$ under the projection to $\Gamma_{i}\times \Gamma_{j}$ has finite index for $0\leq i<j \leq n$.
\end{theorem}

\section{Limit groups over Droms RAAGs are free-by-(torsion-free nilpotent)}

\begin{definition}
A graph is called \emph{triangulated} if it contains no induced copy of $C_{n}$, for $n\geq 4$, where $C_n$ is the circle with $n$ vertices.
\end{definition}

Recall that a RAAG is coherent precisely when its corresponding graph is triangulated, see \cite{Droms}.

\begin{theorem}{\cite[Theorem 2]{Servatius}}
If $G$ is the RAAG corresponding to the graph $X$, the commutator subgroup $G^\prime$ is free if and only if $X$ is triangulated.
\end{theorem}

\begin{lemma}
 Coherent  RAAGs are free-by-(free abelian).
\end{lemma}

\begin{proof}
Let $G$ be a  coherent  RAAG corresponding to the graph $X$. Then, $G^\prime$ is free, the abelianization is $\mathbb{Z}^n$ where $n$ is the number of vertices in the graph $X$ and there is a short exact sequence
\[
1 \to  G^\prime \to G \to \mathbb{Z}^n \to 1.
\]
\end{proof}

A filtration $\{G_{i}\}_{i\geq 1}$ of normal subgroups of a group $G$ has \emph{Property (1)} if $G \slash G_{i}$ is torsion-free and $\bigcap_{i\geq 1} G_{i}=1$, and it has \emph{Property (2)} if $\bigcap_{i\geq 1} G_{i}=1$ and for each finitely generated abelian subgroup $M$ of $G$ there is $i=i(M)$ such that $G_{i} \cap M=1$.

\begin{lemma}\label{Lemma 1}
Let $G$ be a group and $\{G_{i}\}_{i\geq 1}$ be a filtration of normal subgroups of $G$. If $\{G_{i}\}_{i\geq 1}$ has Property (1), it also has Property (2).
\end{lemma}

\begin{proof}
Let $M$ be a finitely generated abelian subgroup of $G$. Since $\bigcap_{i\geq 1} G_{i}=1$, there is $i$ such that $G_{i} \cap M$ is not $M$. The group $M \slash (M \cap G_{i})$ embeds in $G \slash G_{i}$, so it is non-trivial and torsion-free, that is, a finite rank free abelian group. As $M$ has finite Hirsch length, we deduce that $M \cap G_{j}$ is trivial for sufficiently large $j$.
\end{proof}

\begin{lemma}\label{Lemma 2}
Let $H$ be a group, let $G$ be  $\mathbb{Z}^m \times H$ for some $m\in \mathbb{N}$ and let us denote the projection map $G \mapsto H$ by $p$. Suppose that $\{ G_{i} \}_{i\geq 1}$ is a filtration of normal subgroups of $G$ with Property (2). Then,  $\{ p(G_{i}) \}_{i\geq 1}$ is a filtration of normal subgroups of $H$ with Property (2).
\end{lemma}

\begin{proof}
Let us denote $\mathbb{Z}^m $ by $A$ and let us define $H_{i}$ to be $p(G_i)$.  Thus, $AG_{i}= AH_{i}$.
Let $C$  be a finitely generated abelian subgroup of $H$. Note that the group $A(H_{i}\cap C)$ is contained in \[AH_{i}\cap AC= AG_{i} \cap AC= A(G_{i} \cap AC).\] But since the filtration $\{G_{i}\}$ has Property (2), there is $i_{1}= i(AC)$ such that $G_{i_{1}} \cap AC=1$. To sum up, since $A \cap C \subseteq A \cap H=1$, $A(H_{i}\cap C)$ being equal to $A$ implies that $H_{i} \cap C=1$ for  $i\geq i_{1}$.
\end{proof}

\begin{proposition} \label{prop-free1}
Let $G$ be a Droms RAAG and $\{G_{i}\}_{i\geq 1}$ be a filtration of normal subgroups of $G$ with Property (2). Then, for sufficiently large $i_{0}$ the group $G_{i_{0}}$ is free.
\end{proposition}

\begin{proof}
Let $G$ be a Droms RAAG of level $l(G)$. If $l(G)=0$, then $G$ is a free group, so the statement clearly holds. If $l(G)\geq 1$, then $G$ equals $\mathbb{Z}^m \times (K_{1}\ast \cdots \ast K_{k})$ for some $m\in \mathbb{N} \cup \{0\}$ and $K_{1},\dots,K_{k}$ Droms RAAGs of level less than $l(G)$.

Let us denote $\mathbb{Z}^m$ and $K_{1}\ast \cdots \ast K_{k}$ by $A$ and $H$, respectively, and the projection map $G \mapsto H$ by $p$. By hypothesis, there is $i_{1}=i(A)$ such that $A \cap G_{i_{1}}=1$. 

Let us define $H_{i}$ to be $p(G_i)$. From Lemma \ref{Lemma 2} we get that there is $i_{2}\geq 1$ such that $\{H_{i}\}_{i\geq i_{2}}$ is a filtration of $H$ with Property (2), so for $i_{3} = \max \{i_{1},i_{2}\}$ the filtration $\{H_{i}\}_{i\geq i_{3}}$ of normal subgroups of $H$ has Property (2) and $G_{i} \simeq H_{i}$. 

The group $H_{i}$ is a free product of conjugates of $H_{i} \cap K_{j}$ for $j\in \{1,\dots,k\}$ and a free group. For $j\in \{1,\dots, k\}$, $\{ H_{i} \cap K_{j}\}_{i\geq i_{3}}$ is a filtration of normal subgroups of $K_{j}$ with Property (2), so by inductive hypothesis, there is  $s_{j}$ such that $H_{s_{j}} \cap K_{j}$ is free. In conclusion, by taking $s_{0}$ to be $\max \{s_{1},\dots,s_{k}\}$, $H_{s_{0}}$ is a free group.
\end{proof}

\begin{theorem}\label{Theorem 1}
Let $G$ be a Droms RAAG with filtration $\{G_{i}\}_{i \geq 1}$ of normal subgroups such that $G \slash G_{i}$ is torsion-free and $\bigcap_{i\geq 1} G_{i}=1$. Then, for sufficiently large $i_{0}$ the group $G_{i_{0}}$ is free.
\end{theorem}

\begin{proof} It follows from Lemma~\ref{Lemma 1} and Proposition~\ref{prop-free1}.
\end{proof}

In order to show that limit groups over  Droms RAAGs are free-by-(torsion-free nilpotent), we will show that ICE groups over Droms RAAGs are free-by-(torsion-free nilpotent). A limit group over a Droms RAAG is a finitely generated subgroup of an ICE group over a Droms RAAG, so it will also be free-by-(torsion-free nilpotent).

Let us start recalling the definition of ICE groups over Droms RAAGs given in \cite{Montse}.

\begin{definition} \label{defIce}
Let $H$ be a group and $Z \subseteq H$ the centralizer of an element. Then, the group $G= H \ast_{Z} (Z \times \mathbb{Z}^n)$ is said to be obtained from $H$ by an \emph{extension of a centralizer}. An \emph{ICE group over $H$} is a group obtained from $H$ by applying finitely many times the extension of a centralizer construction.  In the classic limit groups, the standard condition that is asked is that the abelian subgroup $Z$ is maximal in one of the vertex groups and indeed this assures 2-acylindricity of the action. However, if one asks $Z$ to be not just abelian but a centralizer of an element, this maximality condition is satisfied.
\end{definition}

\begin{theorem}[\cite{Montse}] \label{Limit-subgroup-ICE}
All ICE groups over Droms RAAGs are limit groups over Droms RAAGs. Moreover, a group is a limit group over a Droms RAAG if and only if it is a finitely generated subgroup of an ICE group over a Droms RAAG.
\end{theorem}

\begin{remark}\label{Remark level}
In \cite[Lemma 7.5]{Montse}  it is proved that any ICE over a Droms RAAG can be obtained by first extending centralizers of central elements, then extending centralizers of elliptic elements and finally considering extensions of centralizers of hyperbolic elements.

If $G$ is a Droms RAAG such that $l(G)=0$, then the centralizer of any non-trivial element  $t \in G$ is the infinite cyclic group  $\langle t_1 \rangle$, where   $\langle t \rangle \subseteq \langle t_1 \rangle$ and $[\langle t_1 \rangle : \langle t \rangle] $ is maximal possible.

 Now let $G$ be a Droms RAAG of the form $\mathbb{Z}^t \times (G_{1}\ast \cdots \ast G_{f})$ for $t,f\in \mathbb{N} \cup \{0\}$ such that $l(G_{i}) \leq l(G)-1$ for $i\in \{1,\dots,f\}$, let $g\in G$ and let us denote the centralizer of $g$ in $G$ by $C_{G}(g)$. 

Firstly, by the result \cite[Lemma 7.5]{Montse} mentioned above, we shall start extending centralizers of central elements. If $g\in \mathbb{Z}^t$, then $C_{G}(g)= G$, so the extension of a centralizer construction gives \[ G \ast_{C_{G}(g)} (C_{G}(g) \times \mathbb{Z}^k)\simeq (\mathbb{Z}^t \times \mathbb{Z}^k) \times (G_{1}\ast \cdots \ast G_{f}),\]
and we again get a Droms RAAG. That is, extending centralizers of central elements in Droms RAAGs leads to another Droms RAAG. Suppose that after extending centralizers of central elements, we have constructed the group $\mathbb{Z}^m \times (G_1 \ast \cdots \ast G_n)$.

Secondly, we extend centralizers of elliptic elements, so we assume that $g \in \mathbb{Z}^m g_0$, $g_0 \neq 1$ and  $g_0$ is an elliptic element in $G_{1}\ast \cdots \ast G_{n}$. Then, $g_0 \in \bigcup_{1 \leq i \leq n} G_i^G$, $C_G(g) = C_G( g_0)$ and we can assume without loss of generality that $g = g_0$. Since extensions of centralizers are conjugate invariant, we can assume  that $g \in G_{i}$ for some $i\in \{1,\dots,n\}$. Thus, $$C_{G}(g)= \mathbb{Z}^m \times C_{G_{i}}(g).$$
Therefore, 
\[ G \ast_{C_{G}(g)} (C_{G}(g) \times \mathbb{Z}^k)\simeq \]\[\mathbb{Z}^m \times \Big{(}G_{1} \ast \cdots \ast G_{i-1} \ast \big{(}G_{i} \ast_{C_{G_{i}}(g)} (\mathbb{Z}^k \times C_{G_{i}}(g))\big{)} \ast G_{i+1} \ast \cdots \ast G_{n} \Big{)},\]
so we get a group of the form $H\colon = \mathbb{Z}^{m} \times (H_1 \ast \cdots \ast H_n)$. Elliptic elements of $H$ lie again in $ \mathbb{Z}^{m} g_0$ with $g_0 \neq 1$ and $g_0$ elliptic in $H_1 \ast \cdots \ast H_n$. Extending centralizers of elliptic elements in groups of this form thus ends up giving again a group of the form $\mathbb{Z}^l \times (K_1 \ast \cdots \ast K_v)$.

 Thirdly, we shall extend the centralizer of a hyperbolic element in a group of the form $K \colon = \mathbb{Z}^l \times (K_1 \ast \cdots \ast K_v)$. Hence, assume that $g \in \mathbb{Z}^l g_0$, $g_0 \neq 1$ and  $g_0$ is a hyperbolic element in $K_{1}\ast \cdots \ast K_{v}$. Then, $g_0 \in  (K_{1}\ast \cdots \ast K_{v}) \setminus (\bigcup_{1 \leq i \leq v} K_i^K)$, $C_K(g) = C_K( g_0)$ and we can assume without loss of generality that $g = g_0$. Hence, $C_{K}(g)=  \mathbb{Z}^l \times {\langle g_1 \rangle}  \simeq \mathbb{Z}^{l+1}$,  where $\langle g \rangle \subseteq \langle g_1 \rangle$ and $[\langle g_1 \rangle : \langle g \rangle] $ is maximal possible, and
\[ K \ast_{C_{K}(g)} (C_{K}(g) \times \mathbb{Z}^k)\simeq \Big{(}
\mathbb{Z}^l \times (K_{1}\ast \cdots \ast K_{v})\Big{)} \ast_{\mathbb{Z}^l \times  \langle g_1 \rangle} (\mathbb{Z}^l \times  \langle  g_1 \rangle \times \mathbb{Z}^k) \simeq \] \[ 
\mathbb{Z}^l \times\Big{(} (K_{1}\ast \cdots \ast K_{v}) \ast_{\mathbb{Z}} \mathbb{Z}^{k+1}  \Big{)}. \]
In this case, we obtain a group of the form $H \colon = \mathbb{Z}^{l} \times (A \ast_{B} C)$. Hyperbolic elements of this group are elements $g\in \mathbb{Z}^{l}g_0$ with $g_0$ hyperbolic in $A\ast_{B} C$, that is, $g_0 \in (A\ast_B C) \setminus (A^H \cup C^H)$ and extending centralizers of hyperbolic elements in groups of this form yields a group with the same structure. 

Thus, the class of ICE groups over Droms RAAGs can be given the following hierarchical structure. 
\end{remark}

\begin{definition} 
Assume that $G$ is a Droms RAAG of level 0, that is, $G$ is a free group. An ICE group over $G$ of  altitude  0 is precisely $G$. An ICE group over $G$ of   altitude  $k\geq 1$ is an amalgamated free product over $\mathbb{Z}^{n_0}$ of an ICE group over $G$ of  altitude  $\leq k-1$ and a free abelian group $\mathbb{Z}^{n_0} \times \mathbb{Z}^{m_0 }$, where $\mathbb{Z}^{n_0}$ embeds in $\mathbb{Z}^{n_0} \times \mathbb{Z}^{m_0}$ naturally.

Assume that $G$ is a Droms RAAG of level $l\geq 1$, that is,  $G= \mathbb Z^m \times (G_1 \ast \cdots \ast G_n)$ where each $G_i$ is a Droms RAAG of level less than $l$. An ICE group over $G$ of  altitude  $0$ is of the form $\mathbb Z^{m^\prime} \times (H_1 \ast \cdots \ast H_n)$ where $m^\prime \geq m$ and $H_i$ is an ICE group over $G_i$ for $i\in \{1,\dots, n\}$.

An ICE group over $G$ of  altitude  $k\geq 1$ is an amalgamated free product over $\mathbb{Z}^{n_0}$ of an ICE group over $G$ of  altitude $\leq k-1$ and a free abelian group $\mathbb{Z}^{n_0} \times \mathbb{Z}^{m_0}$, where $\mathbb{Z}^{n_0}$ embeds in $\mathbb{Z}^{n_0} \times \mathbb{Z}^{m_0}$ naturally.
\end{definition}

\begin{remark}
The term used in the literature for the altitude of ICE groups over free groups is \emph{level}. However, since level has been used when defining Droms RAAGs hierarchically, we have decided to use \emph{altitude} for ICE groups over Droms RAAGs.
\end{remark}

\begin{theorem} \label{teo-free1}
Let $G$ be a Droms RAAG and let $K$ be an ICE group over $G$. Let $\{{K }_{i}\}_{i \geq 1}$ be a filtration of normal subgroups of $K$ with Property (2). Then, for sufficiently large $i_{0}$ the group ${K }_{i_{0}}$ is free.
\end{theorem}

\begin{proof}
We prove it by induction on the level $l(G)$ of $G$ as a Droms RAAG. Suppose that $l(G)=0$. Then, if $K $ has   altitude  $0$, $K $ is precisely $G$ which is a free group, and the result is obvious. If $K $ has   altitude  $\geq 1$, then $K $ is of the form
\[ H \ast_{\mathbb{Z}^{n_0}} (\mathbb{Z}^{m_0} \times \mathbb{Z}^{n_0}),\] and $H$ is an ICE group over $G$ of smaller  altitude than $K$. Since $K_{i}$ is a normal subgroup of $K $, $K_{i}$ inherits a graph of groups decomposition where the vertex groups are of the form \[ (K_{i}\cap H)^{g_{\alpha_{j}}} \quad \text{or} \quad (K_{i}\cap (\mathbb{Z}^{m_0} \times \mathbb{Z}^{n_0}))^{g_{\alpha_{j}}}, \quad \text{for} \quad g_{\alpha_{j}}\in K,\]
and the edge groups are of the form \[ (K_{i} \cap \mathbb{Z}^{n_0})^{g_{\beta_{j}}} \quad \text{for} \quad g_{\beta_{j}} \in K.\]
Note that $\{K_{i} \cap H \}_{i\geq 1}$ is a filtration of $H$ with Property (2), so by inductive hypothesis, there is $i_{1}$ such that $K_{i_{1}} \cap H$ is free. In addition, the filtration $\{K_{i}\}_{i\geq 1}$ has Property (2), so there is $i_{2}= i(\mathbb{Z}^{m_{0}} \times \mathbb{Z}^{n_0})$ such that $K_{i_{2}} \cap (\mathbb{Z}^{m_0} \times \mathbb{Z}^{ n_0})$ is trivial. In particular, $K_{i_{2}} \cap \mathbb{Z}^{n_0}$ is also trivial. Therefore, taking $i_{0}$ to be $\max \{i_{1}, i_{2}\}$, $K_{i_{0}}$ is free.

Now suppose that $l(G) \geq 1$. Then, $G$ is of the form \[ \mathbb{Z}^m \times (G_{1}\ast \cdots \ast G_{n}),\] for some $m\in \mathbb{N} \cup \{0\}$ and $G_{i}$ is a Droms RAAG with $l(G_{i}) \leq l(G)-1$ for each $i\in \{1,\dots,n\}$. If $K$ has altitude $0$ as an ICE group over $G$, then \[ K= \mathbb{Z}^{m^\prime} \times (H_{1}\ast \cdots \ast H_{n}),\] where $m^\prime \geq m$ and $H_{i}$ is an ICE group over $G_{i}$ for $i\in \{1,\dots,n\}$. Let us denote the projection map $K \mapsto H_{1}\ast \cdots \ast H_{n}$ by $p$.

By hypothesis, there is $i_{1}= i(\mathbb{Z}^{m^\prime})$ such that $\mathbb{Z}^{m^\prime} \cap K_{i_{1}}=1$. In addition, by Lemma \ref{Lemma 2},  $\{ N_{i}\}_{i\geq 1}$ is a filtration of normal subgroups of $H_{1}\ast \cdots \ast H_{n}$ with Property (2), where $N_{i}= p(K_{i})$. As a consequence, $N_{i} \simeq K_{i}$ for $i \geq i_1$. 

Note that the group $N_{i}$ is a free product of conjugates of $N_{i} \cap H_{j}$ for $j\in \{1,\dots,n\}$ and a free group. For $j\in \{1,\dots,n\}$, $\{N_{i} \cap H_{j}\}_{i\geq 1}$ is a filtration of normal subgroups of $H_{j}$ with Property (2), so by inductive hypothesis, there is $r_{j}$ such that $N_{r_{j}} \cap H_{j}$ is free. In conclusion, taking $i_{0}$ to be $\max \{i_1, r_{1},\dots,r_{n} \}$, $N_{i_{0}}$ is a free group.

Finally, suppose that $K$ is an ICE group over $G$ of   altitude $k \geq 1$. Then, $K$ is an amalgamated free product over $\mathbb{Z}^{n_0}$ of an ICE group over $G$ of   altitude $\leq k-1$ and a free abelian group $\mathbb{Z}^{ n_0} \times \mathbb{Z}^{ m_0}$. This case may be treated as the case when $K$ is an ICE group of   altitude greater than $0$ over a free group.
\end{proof}

Recall that $\gamma_i(G)$ denotes the lower central series of a group $G$. That is, \[\gamma_1(G) = G \quad \text{and} \quad \gamma_{i+1} (G) = [ G, \gamma_i(G)].\] We denote by $tor(G)$ the set of torsion elements of $G$. In the case when $G$ is finitely generated nilpotent, then $\langle tor(G) \rangle$ is the maximal finite subgroup of $G$.

\begin{theorem} \label{teo-torsion-free1} 
Let $G$ be  a Droms RAAG, let $K$ be an ICE group over $G$ and define $K_{i}$ to be $K_{i} \slash \gamma_{i}(K)= \langle \text{tor}(K \slash \gamma_{i}(K)) \rangle$. Then, $K_{i+1} < K_{i}$, $K_{i}$ is normal in $K$, $K \slash K_{i}$ is torsion-free nilpotent and $\bigcap_{i\geq 1} K_{i}=1$.
\end{theorem}

\begin{proof}
By construction, $K_i \slash \gamma_i(K)$ is a characteristic subgroup of $K\slash \gamma_i(K)$, so $K_i$ is normal in $K$ and $K_{i+1} < K_i$. It remains to show that $\bigcap_{i} K_i = 1$.
Suppose that $k \in (\bigcap_{i} K_i) \setminus \{ 1 \}$. Since $K$ is a limit group over $G$, there is a homomorphism $\varphi \colon K \mapsto G$ such that $\varphi(k) \not= 1$. Let $G_0 = \text{im}(\varphi)$, so $G_0$ is a Droms RAAG.
 By \cite[Theorem 6.4]{Wade}, we have that
 $G_0 \slash \gamma_{i}(G_0)$ is torsion-free, hence $\varphi(K_i) = \varphi(\gamma_i(K)) $. Then $\varphi(k) \in \bigcap_i \varphi(K_i) = \bigcap_i  \varphi(\gamma_i(K)) \subseteq \bigcap_i \gamma_i(G) = 1$, a contradiction.
\end{proof}

\begin{corollary} \label{free-nilpotent}
Every ICE group over a Droms RAAG is free-by-(torsion-free nilpotent).
\end{corollary}

\begin{proof}
It follows from Lemma \ref{Lemma 1}, Theorem \ref{teo-free1} and  Theorem \ref{teo-torsion-free1}. 
\end{proof}

\begin{proposition} \label{free-by-(torsion-free nilpotent)}
Every limit group over a Droms RAAG is free-by-(torsion-free nilpotent).
\end{proposition}

\begin{proof} It follows from Theorem \ref{Limit-subgroup-ICE}
and Corollary \ref{free-nilpotent}.
\end{proof}

\section{Corollaries on subdirect products of limit groups over Droms RAAGs}
\label{Section 5}

In this section we aim to prove Theorem \ref{FPs}.  Theorem \ref{FPs} is a generalisation of \cite[Theorem 11]{Desi}.

\begin{theorem} \cite[Theorem 11]{Desi} \label{old11}
Let $G_1, \dots, G_m$ be non-abelian limit groups over free groups and let \[S 
 < G \colon = G_1 \times \cdots \times G_m\] be a finitely generated subdirect product intersecting each factor $G_i$ non-trivially. Suppose further that, for some fixed natural number $s \in \{ 2, \ldots, m \}$, for every subgroup $S_0$ of finite index in $S$ the homology group $H_i (S_0 , \mathbb{Q})$ is finite dimensional (over $\mathbb{Q}$) for all
$i \leq  s$. Then for every canonical projection \[p_{j_1, \dots,  j_s} \colon S \mapsto G_{j_1} \times \cdots \times G_{ j_ s}\]
the index of $p_{j_1, \dots, j_s} (S)$ in  $G_{j_1} \times \cdots \times G_{ j_s}$ is finite.
\end{theorem}

We observe that in \cite{Desi} the domain of $p_{j_1, \dots, j_s}$ is defined to be $G$ and not as we stated it above, but in \cite{Bridson2} and \cite{Bridson} it is $S$. It looks more convenient to stick to the latter.

The starting point in the proof of Theorem \ref{FPs} is that the proof of Theorem \ref{old11} applies in bigger generality. We explain this in the following result.

\begin{theorem} \label{new1}
Let $2 \leq s \leq m$ be integers and $S < G_1 \times \dots \times G_m$ be a finitely generated subdirect product such that:\\[5pt]
(1) there exist normal free subgroups $L_i$ in $G_i$ with $Q_i \colon = G_i \slash L_i$ nilpotent;\\[5pt]
(2) each $G_i$ is finitely presented;\\[5pt]
(3) for every $1 \leq j_1 < \dots < j_s \leq m$ if $M_{j_1, \dots, j_s} $ is a subgroup of infinite index in the group $G_{j_1} \times \cdots \times G_{j_s}$, then there exists $i \leq s$   and a subgroup $M_{j_1, \ldots, j_s}^0$ of finite index in $M_{j_1, \ldots, j_s}$ such that $H_i(M_{j_1, \ldots, j_s}^0, \mathbb{Q})$ is infinite dimensional over $\mathbb{Q}$;\\[5pt]
(4) for every $1 \leq j_1 < \dots < j_s \leq m$ and every subgroup $H_{j_i}$ of finite index in $G_{j_i}$ we have that if a subdirect product $M_{j_1, \dots, j_s} \subseteq H_{j_1} \times \cdots \times H_{j_s}$ is finitely presented, then there is a subgroup $K_{j_i}$ of finite index in $H_{j_i}$ and $N_{j_i}$ a normal subgroup of  $H_{j_i}$ such that  $K_{j_i}\slash N_{j_i}$ is nilpotent and $N_{j_1} \times \cdots \times N_{j_s} \subseteq M_{j_1, \ldots, j_s }$;\\[5pt]
(5) $S$ virtually surjects on pairs;\\[5pt]
(6) $L \colon =L_1 \times \cdots \times L_m \subseteq S$;\\[5pt]
(7) for every subgroup $S_0$ of finite index in $S$, the homology group $H_i (S_0 , \mathbb{Q})$ is finite dimensional (over $\mathbb{Q}$) for all
$ i \leq  s$. 

Then for every canonical projection
\[p_{j_1, \dots,  j_s} \colon S \mapsto G_{j_1} \times \cdots \times G_{ j_ s},\]
the index of $p_{j_1, \ldots, j_s} (S)$ in  $G_{j_1} \times \cdots \times G_{ j_s}$ is finite.
\end{theorem}

\begin{proof} We will prove the result by induction on $s$. For $s = 2$, this is condition 5.

Suppose that $s \geq 3$. We divide the proof in several steps.

1) Set $Q \colon = S\slash L < Q_1 \times \cdots \times Q_m$. Consider the Lyndon--Hochschild--Serre spectral sequence \[E^2_{i,j} = H_i(Q, H_j(L, \mathbb{Q}))\] that converges to $H_{i+j}(S,\mathbb{Q})$.

Note that since each $L_i$ is a free group, for every $t \leq m$ we have that
\begin{equation} \label{old2} H_t(L, \mathbb{Q}) \simeq \bigoplus_{1 \leq j_1 < j_2 < \dots < j_t \leq m} H_1(L_{j_1}, \mathbb{Q}) \otimes_{\mathbb{Q}} \dots \otimes_{\mathbb{Q}} H_1(L_{j_t}, \mathbb{Q}),\end{equation}
where each summand is $Q$-invariant, where the $Q$-action is induced by conjugation. Thus,
\[E^2_{0,s} = H_0(Q, H_s(L, \mathbb{Q})) \simeq \bigoplus_{1 \leq j_1 < j_2 < \ldots < j_s \leq m} \Big{(} H_1(L_{j_1}, \mathbb{Q}) \otimes \ldots \otimes H_1(L_{j_s}, \mathbb{Q})  \otimes_{\mathbb{Q}Q} \mathbb{Q} \Big{)}.\]

 By inductive hypothesis, $S$ virtually surjects onto $s-1$ factors. This implies that $ H_j(L, \mathbb{Q})$ is a finitely generated  $\mathbb{Z}Q$-module for $j \leq s-1$, and hence,
\begin{equation} \label{qq1}  E^2_{i,j} \hbox{ is finite dimensional over } \mathbb{Q} \hbox{ for every }j \leq s-1. \end{equation}

For $ i \geq 2$, we have that $s+1 - i \leq s-1$, so by (\ref{qq1}), $E_{i,s+1 - i}^i $ is finite dimensional over $\mathbb{Q}$. Hence, for all differential maps $d_{i, s+1-i }^i \colon E_{i,s+1 - i}^i \mapsto E_{0,s}^i$, the subgroup $\text{im}(d_{i,  s+ 1 - i }^i)$ is finite dimensional over $\mathbb{Q}$. Thus, \[E_{0,s}^{ i+1} = \ker ( d_{0,s}^i) \slash \text{im}(d_{i,  s+1-i }^i) = E_{0,s}^i \slash \text{im}(d_{i,  s+1-i }^i)\] is finite dimensional over $\mathbb{Q}$ if and only if $ E_{0,s}^i$ is finite dimensional over $\mathbb{Q}$. This implies that $E^2_{0,s}$  is finite dimensional over $\mathbb{Q}$ if and only if $E^{ \infty}_{0,s}$ is finite dimensional over $\mathbb{Q}$. Combining this with the convergence of the spectral sequence and (\ref{qq1}) we deduce that
\[ H_s(S, \mathbb{Q}) \hbox{ is finite dimensional over }\mathbb{Q}\hbox{ if and only if }\]\[E^2_{0,s} = H_0(Q, H_s(L,\mathbb{Q})) \hbox{ is finite dimensional over }\mathbb{Q}.\]

The condition that $H_s(S,\mathbb{Q})$ is finite dimensional over $\mathbb{Q}$ implies that
for each $1 \leq j_1 < j_2 < \ldots < j_s \leq m$,
\begin{equation} \label{old1} \dim_{\mathbb{Q}} \big{(} H_1(L_{j_1}, \mathbb{Q}) \otimes \dots \otimes H_1(L_{j_s}, \mathbb{Q}) \otimes_{\mathbb{Q}Q} \mathbb{Q}\big{)} = \end{equation} \[\dim_{\mathbb{Q}} \big{(}H_1(L_{j_1}, \mathbb{Q}) \otimes \dots \otimes H_1(L_{j_s}, \mathbb{Q})  \otimes_{\mathbb{Q} h_{j_1, \ldots, j_s}(Q)} \mathbb{Q} \big{)}< \infty,
\]
where
$h_{j_1, \dots, j_s} \colon Q \mapsto Q_{j_1} \times \cdots \times Q_{j_s}$ is the canonical projection.

2) We consider $\widetilde{S}$ a subgroup of finite index in $p_{j_1, \dots, j_s}(S)$ that contains $\widetilde{L} \colon = L_{j_1} \times \cdots \times L_{j_s}$ and we set $\widetilde{Q} \colon =\widetilde{S} \slash \widetilde{L}$. Then, the Lyndon--Hochschild--Serre spectral sequence \[\widetilde{E}^2_{i,j} = H_i(\widetilde{Q}, H_j( \widetilde{L}, \mathbb{Q}))\] converges to $H_{i+ j}(\widetilde{S}, \mathbb{Q})$. By (\ref{old2}) and (\ref{old1}) we deduce that $\dim_{\mathbb{Q}} \widetilde{E}^2_{i,j} < \infty$   for $j \leq s-1$ and  $\dim_{\mathbb{Q}} \widetilde{E}^2_{0,s}< \infty$. Then, by the convergence of the spectral sequence, we deduce that  $\dim_{\mathbb{Q}} H_i(\widetilde{S},\mathbb{Q}) < \infty$ for  $i \leq s$.

3) Now we  consider $\widetilde{S}_0$ an arbitrary subgroup of finite index in $p_{j_1, \dots, j_s}(S)$. We view $p_{j_1, \dots, j_s}(S)$ as a subdirect product of $G_{j_1} \times \cdots \times G_{j_s}$. As $S$ virtually surjects on pairs we have that $p_{j_1, \dots, j_s}(S)$ virtually surjects on pairs, so by \cite[Theorem A]{Bridson2}, $p_{j_1, \dots, j_s}(S)$ is finitely presented. Hence $\widetilde{S}_0$ is finitely presented. Then, by condition 4 applied to $\widetilde{S}_0$ considered as a subdirect product of $p_{j_1}(\widetilde{S}_0) \times \cdots \times p_{j_s}(\widetilde{S}_0)$, there is a subgroup of finite index $K_{j_i}$ in $p_{j_i}(\widetilde{S}_0)$ such that for sufficiently big $t$ we have that $ \gamma_t(K_{j_i}) \subseteq \widetilde{S}_0$. By condition 1, we can choose $t$ sufficiently big so that $\gamma_t(G_k) \subseteq L_k$ is free for every $ 1 \leq k \leq m$. In particular, $\gamma_t(K_{j_i}) $ is free.

 Note that each $K_{j_i}$ can be chosen normal  in $p_{j_i}(\widetilde{S}_0)$ (by substituting if necessary $K_{j_i}$ with its core in $p_{j_i}(\widetilde{S}_0)$)  and that
$p_{j_i}(\widetilde{S}_0)$ has finite index in $G_{j_i} = p_{j_i}(S)$, hence $K_{j_i}$ has finite index in $G_{j_i}$ for $ 1 \leq i \leq s$. Thus
$\widetilde{S}_0 \cap ( K_{j_1} \times \ldots \times K_{j_s} )$ has finite index in $\widetilde{S}_0$.
 By construction  $\widetilde{S}_0$ is a subgroup of finite index in $p_{j_1, \dots, j_s}(S)$ and so $\widetilde{S}_0 \cap ( K_{j_1} \times \ldots \times K_{j_s} )$ has finite index in $p_{j_1, \dots, j_s}(S)$. We set $S_0 = p_{j_1, \ldots, j_s}^{-1} (\widetilde{S}_0 \cap ( K_{j_1} \times \ldots \times K_{j_s} ))$ and note it is a subgroup of finite index in $S$. Note that $p_{j_i} (S_0) \subseteq K_{j_i}$ for $ 1 \leq i \leq s$.

 By condition 7, $H_i (S_0 , \mathbb{Q})$ is finite dimensional (over $\mathbb{Q}$) for all $ i \leq  s.$  Applying Step 2 for $S$ substituted with $S_0$, $G_j$ substituted with
 $p_j(S_0)$
  and $L_j$ substituted with $\gamma_t(p_j(S_0))$, we deduce that for every subgroup $B$ of  finite index in $p_{j_1, \dots, j_s}(S_0) = \widetilde{S}_0 \cap ( K_{j_1} \times \ldots \times K_{j_s} )$ such that $B$ contains $\gamma_t(p_{j_1}(S_0)) \times \ldots \times \gamma_t(p_{j_s}(S_0))$
 we have that 
\[\dim_{\mathbb{Q}} H_i (B, \mathbb{Q}) < \infty \quad \text{for} \quad i \leq s.\]
We apply this for $B = p_{j_1, \dots, j_s}(S_0) = \widetilde{S}_0 \cap ( K_{j_1} \times \ldots \times K_{j_s} )$ and note that $B$ is a normal subgroup of finite index in $\widetilde{S}_0$. Hence the Lyndon-Hoschild-Serre spectral sequence $H_i(\widetilde{S}_0/B, H_j (B, \mathbb{Q}))$ that converges to $H_{i+j}(\widetilde{S}_0, \mathbb{Q})$ implies that 
\[\dim_{\mathbb{Q}} H_i (\widetilde{S}_0, \mathbb{Q}) < \infty \quad \text{for} \quad i \leq s.\]
But this combined with condition 3 and the fact that  $\widetilde{S}_0$ is an arbitrary subgroup of finite index in $p_{j_1, \dots, j_s}(S)$ implies that $p_{j_1, \dots, j_s} (S)$ has finite index in $G_{j_1} \times \cdots \times G_{j_s}$. 
\end{proof}

Let $n \geq 0$, $R$ be an associative ring with identity and let $M$ be a (right) $R$-module. Then $M$ is  of \emph{(homological) type $FP_n$} if there is a projective resolution \[{\mathcal P} \colon \dots \mapsto P_n \mapsto \dots \mapsto P_i \mapsto P_{i-1} \mapsto \dots \mapsto P_0 \mapsto M \mapsto 0\] such that each projective $R$-module $P_i$ is finitely generated for $0 \leq i \leq n$. A group $G$ is of type $FP_n$ over a  commutative ring $S$ if the trivial $S G$-module $S$ is of type $FP_n$ over the group algebra $S G$.

\begin{lemma} \cite[Lemma 6]{Desi} \label{qq21} Let $Q_1, \ldots, Q_i$ be finitely generated nilpotent groups and $V_j$ be a finitely
generated $\mathbb{Q} Q_j$-module such that $V_j$ contains a cyclic non-zero free $\mathbb{Q} Q_j$-submodule
$W_j$ for $ 1 \leq j \leq i$. Suppose that $\widetilde{Q}$
is a subgroup of $Q \colon = Q_1 \times \cdots \times Q_i$  such that $V_1 \otimes_{\mathbb{Q} } \cdots \otimes_{\mathbb{Q} } V_i$ is finitely generated as a $\mathbb{Q} \widetilde{Q}$-module. Then, $\widetilde{Q}$ has finite index in $Q$.
 \end{lemma} 

\begin{proposition} \cite[Proposition 7]{Desi} \label{qq32}
Let $G$ be a group of negative Euler characteristic $\chi(G)$ such that the trivial $\mathbb{Q} G$-module $\mathbb{Q}$ has a free resolution with finitely generated modules and finite length. Then, for any normal subgroup $M$ of $G$ such that $Q \colon =G\slash M$ is torsion-free nilpotent and $M$ is free, the $\mathbb{Q} Q$-module $V = M\slash [M,M] \otimes_{\mathbb{Z}} \mathbb{Q}$ has a non-zero free
$\mathbb{Q} Q$-submodule, where $Q$ acts on $M=M\slash [M, M]$ via conjugation.
\end{proposition}

\begin{theorem} \label{qq43}Let $S  < G_1 \times \cdots \times G_m$ be a finitely generated subdirect product such that:\\[5pt]
(1) for each $ 1 \leq i \leq m$ the trivial $\mathbb{Q} G_i$-module $\mathbb{Q}$ has a free resolution of finite length with finitely generated modules;\\[5pt]
(2) for each $ 1 \leq i \leq m$ there is a normal free subgroup $L_i$ of $G_i$ such that $G_i\slash L_i$ is torsion-free nilpotent;\\[5pt]
(3) $L_1 \times \cdots \times L_m \subseteq S$;\\[5pt]
(4) $S$ is of type $FP_s$ over $\mathbb{Q}$;\\[5pt]
(5) $\chi(G_{i})<0$.

Then, for every canonical projection
$p_{j_1, \ldots,  j_s} \colon S \mapsto G_{j_1} \times \cdots \times G_{ j_ s}$
the index of $p_{j_1, \ldots, j_s} (S)$ in  $G_{j_1} \times \cdots \times G_{ j_s}$ is finite.
\end{theorem}

\begin{proof} 
Note that since $G_i$ is of type $FP_{\infty}$ over $\mathbb{Q}$, it is $FP_1$ over $\mathbb{Q}$ and so $G_i$ is finitely generated. By Proposition \ref{qq32}, $V_i \colon = (L_i/[L_i, L_i]) \otimes_{\mathbb{Z}} \mathbb{Q}$ has a non-zero free
$\mathbb{Q} Q_i$-submodule, where $Q_i \colon = G_i/ L_i$ acts via conjugation.

Let \[{\mathcal F} \colon \ldots \to F_i \to F_{i-1} \to \ldots \to F_0 \to \mathbb{Q} \to 0\]
be a free resolution of the trivial $\mathbb{Q} S$-module $\mathbb{Q}$ with $F_i$ finitely generated for $i \leq s$. We define $L$ to be  $L_1 \times \cdots \times L_m$ and since each $L_i$ is free, by the K$\ddot{\textrm{u}}$nneth formula, 
\[
H_s(L, \mathbb{Q}) = \bigoplus_{1 \leq j_1 < \dots < j_s \leq m} L_{j_1}\slash [L_{j_1}, L_{j_1}] \otimes_{\mathbb{Z}} \dots \otimes_{\mathbb{Z}}  L_{j_s}\slash [L_{j_s}, L_{j_s}] \otimes_{\mathbb{Z}} \mathbb{Q} \simeq\]\[ \bigoplus_{1 \leq j_1 < \dots < j_s \leq m} V_{j_1} \otimes_{\mathbb{Q}} \cdots \otimes_{\mathbb{Q}} V_{j_s}. \]
Note that
$$H_s(L, \mathbb{Q}) \simeq H_s (\mathcal{F} \otimes_{\mathbb{Q}L} \mathbb{Q}) = \ker(d_s)/ \text{im}(d_{s+1}),$$
where $d_j \colon F_j  \otimes_{\mathbb{Q}L} \mathbb{Q} \mapsto F_{j-1}  \otimes_{\mathbb{Q}L} \mathbb{Q}$ is the differential of $\mathcal{F} \otimes_{\mathbb{Q}L} \mathbb{Q}$. Since $F_s  \otimes_{\mathbb{Q}L} \mathbb{Q} $ is a finitely generated $\mathbb{Q}\widehat{Q}$-module, where $\widehat{Q} \colon = S/ L$ is a finitely generated nilpotent group and $\mathbb{Q}\widehat{Q}$ is a Noetherian ring, we deduce that $\ker(d_s)$ is a finitely generated $\mathbb{Q}\widehat{Q}$-module. In particular, $H_s(L, \mathbb{Q})$ and $V_{j_1} \otimes_{\mathbb{Q}} \cdots \otimes_{\mathbb{Q}} V_{j_s}$ are finitely generated $\mathbb{Q} \widehat{Q}$-modules. Note that the action of $\widehat{Q}$ on $V_{j_1} \otimes_{\mathbb{Q}} \cdots \otimes_{\mathbb{Q}} V_{j_s}$  factors through  $\widetilde{Q} \colon = p_{j_1, \ldots, j_s} (S) \slash L_{j_1} \times \cdots \times L_{j_s}$. Then, by Lemma \ref{qq21},  $\widetilde{Q}$ has finite index in $Q_{j_1} \times \cdots \times Q_{j_s}$. This is equivalent to $p_{j_1, \ldots, j_s} (S)$ having finite index in  $G_{j_1} \times \cdots \times G_{ j_s}$.

\end{proof}

The next step is to prove that limit groups over Droms RAAGs satisfy the conditions of Theorem \ref{qq43}.

\begin{lemma} \label{Euler-char} Let $G$ be a Droms RAAG and $\Gamma$ a limit group over $G$. Then, $\chi(\Gamma)\leq 0$. Furthermore, $\chi(\Gamma)=0$ if and only if
 $\Gamma$ has non-trivial center. The latter happens precisely when  $\Gamma = \mathbb{Z}^{l_0} \times \Lambda_0$ for some $l_0 \geq 1$ and with $\Lambda_0$ a centerless subgroup of $\Gamma$.
\end{lemma}

\begin{proof}
Let us prove it by induction on the level  $l(G)$ of $G$. If $l(G)=0$, then $G$ is a free group, so the result follows from \cite[Lemma 5]{Desi}.

Now assume that $l(G)\geq 1$. Then, $G$ equals \[ \mathbb{Z}^m \times (G_{1}\ast \cdots \ast G_{n}),\] where $m\in \mathbb{N} \cup \{0\}$ and $G_{i}$ is a Droms RAAG such that $l(G_{i}) \leq l(G)-1$ for $i\in \{1,\dots,n\}$. From Proposition \ref{Height} we get that $\Gamma$ is of the form $\mathbb{Z}^l \times \Lambda$ where $\Lambda$ is a limit group over $G_{1}\ast \cdots \ast G_{n}$ and if $m=0$, then $l=0$. Therefore, \[ \chi(\Gamma)= \chi(\mathbb{Z}^l) \chi(\Lambda),\]
so if $l\geq 1$, then $\chi(\Gamma)=0$. Let us compute $\chi(\Lambda)$. If the height of $\Lambda$ is $0$, i.e. $h(\Lambda)=0$, then \[ \Lambda= A_{1}\ast \cdots \ast A_{j},\] where for each $t\in \{1,\dots,j\}$ $A_{t}$ is a limit group over $G_{i}$ for some $i\in \{1,\dots,n\}$. Hence, \[ \chi(\Lambda)= \sum_{t\in \{1,\dots,j\} }\chi(A_{t}) -(j-1),\] so applying the inductive hypothesis, we get that $\chi(\Lambda) \leq 1-j$. If $j\geq 2$, then $\chi(\Lambda)< 0$. If $j=1$, $\chi(\Lambda)=\chi(A_{1})$ and $A_{1}$ is a limit group over $G_{i}$. Thus, by induction, $\chi(A_{1})\leq 0$ and $\chi(A_{1})=0$ if and only if $A_{1}$ has non-trivial center.  In this case the center $Z(A_1)$ is a direct factor of $A_1$.

If $h(\Lambda)\geq 1$, then $\Lambda$ acts cocompactly on a tree $T$ where the edge stabilizers are cyclic and the vertex groups are limit groups over $G_1\ast \cdots \ast G_n$ of height at most $h(\Lambda)-1$. Moreover, at least one vertex group $H_{v_{0}}$ has trivial center and so by inductive hypothesis, $\chi(H_{v_{0}}) < 0$.
If $X$ is the quotient graph $T \slash \Lambda$,
 \[ \chi(\Lambda)= \sum_{v\in V(X)} \chi(H_{v}) - \sum_{e\in E(X)}
  \chi(H_{e})  \leq  \sum_{v\in V(X)} \chi(H_{v} ) \leq \chi (H_{v_{0}} ) <0.\]
\end{proof} 

\begin{lemma}
Let $\Gamma$ be a limit group over a Droms RAAG. Then, the trivial $\mathbb{Q}\Gamma$-module $\mathbb{Q}$ has a free resolution with finitely generated modules and of finite length.
\end{lemma}

\begin{proof}
Limit groups over Droms RAAGs are of type $FP_{\infty}$ over $\mathbb{Z}$ (and so over $\mathbb{Q}$) and of finite cohomological dimension.
\end{proof}

\begin{proof}[Proof of Theorem \ref{FPs}]

By \cite[Theorem 6.2]{Jone}, $G_i\slash (S \cap G_i)$ is virtually nilpotent for every $i\in \{1,\dots,m\}$.
By substituting $G_i$ and $S$ with subgroups of finite index if necessary we can assume that
  $G_i\slash (S \cap G_i)$ is torsion-free nilpotent.
 By Proposition~\ref{free-by-(torsion-free nilpotent)},  there is a  normal free subgroup $M_i$ de $G_i$ such that $G_i / M_i$ is torsion-free nilpotent, hence for $L_i = S \cap M_i$ we have that  $L_i$ is a normal subgroup of $G_i$ such that $L_i$ is free, $G_i\slash L_i$ is torsion-free nilpotent and $L_1 \times \cdots \times L_{m} \subseteq S$. Note that  conditions 1 and 6 from Theorem \ref{new1} hold. The other conditions from Theorem \ref{new1} hold when each $G_i$ is a limit group over a Droms RAAG with trivial center: condition 2 is \cite[Corollary 7.8]{Montse},  condition 3 is Theorem \ref{Theorem 8.1}, condition 5 is Theorem \ref{Theorem 10.1}.

Observe that a finite index subgroup of a limit group over a Droms RAAG is a limit group over a Droms RAAG. Hence condition 4 follows from \cite[Theorem 6.2]{Jone}. Finally, condition 7 is assumed in the statement.
\end{proof}

\section{On the $L^2$-Betti numbers and volume gradients of limit groups over Droms RAAGs and their subdirect products}
\label{Section 6}

The aim of this section is to study the growth of homology groups and the volume gradients for limit groups over Droms RAAGs and for finitely presented residually Droms RAAGs, following the paper \cite{Desi2}. Some of the results concerning limit groups over Droms RAAGs hold in a more general setting, more precisely, they also hold for limit groups over coherent RAAGs. Thus, these results (see Theorem \ref{Theorem C} and Theorem \ref{Theorem D}) will be stated for limit groups over coherent RAAGs.
However, the results for finitely presented residually Droms RAAGs make use of Theorem \ref{Theorem 10.1} and in \cite{Jone-Montse} it was shown that this no longer holds for coherent RAAGs. Thus, Theorem \ref{Theorem E} and Theorem \ref{Theorem F} are stated just for residually Droms RAAGs.

In order to study the homology growth and volume gradients, we work with \emph{exhausting normal chains}: a chain $(B_{n})$ of normal subgroups of finite index such that $B_{n+1} \subseteq B_{n}$ and $\bigcap_{n} B_{n}=1$. Note that if a group is residually finite, then it has an exhausting normal chain. In particular, limit groups over coherent RAAGs have exhausting normal chains.

 A group $G$ is of homotopical type $F_m$ if there is a classifying space $K(G,1)$ with finite $m$-skeleton.
Given a group $G$ of homotopical type $F_{m}$, $\text{vol}_{m}(G)$ is defined to be the least number of $m$-cells among all classifying spaces $K(G,1)$ with finite $m$-skeleton. For instance, $\text{vol}_{1}(G)$ equals $d(G)$, where $d(G)$ is the minimal number of generators of $G$.

One of the aims of this section is to prove Theorem \ref{Theorem C} and Theorem \ref{Theorem D}.
These two results are proved in \cite{Desi2} for limit groups over free groups via more technical results that make use of \emph{slowness} of limit groups over free groups (see Section \ref{subsection slow} for the definition). In \cite{Desi2} it is shown that limit groups over free groups are slow above dimension $1$ and hence are  $K$-slow above dimension $1$  for any field $K$. Thus, the key point is to show that limit groups over coherent RAAGs are also slow above dimension $1$  (see Section \ref{coherent}). 

\begin{theorem}\cite[Theorem D]{Desi2}\label{Theorem D2}
If a residually finite group $G$ of type $F$ is slow above dimension $1$, then with respect to every exhausting normal chain $(B_{n})$,\\[5pt]
(1) Rank gradient: \[ RG(G, (B_{n}))= \lim_{n\to \infty} \frac{d(B_{n})}{[G \colon B_{n}]}= -\chi(G).\]
(2) Deficiency gradient: \[ DG(G, (B_{n}))= \lim_{n\to \infty} \frac{\text{def}(B_{n})}{[G \colon B_{n}]} = \chi(G).\]
\end{theorem}

\begin{lemma}\cite[Lemma 5.2]{Desi2} \label{homology-slow-limit-free} 
Let $K$ be a field and let $G$ be a residually finite group of type $F$ with an exhausting normal chain $(B_n)$. If $G$ is $K$-slow above dimension $1$, then \[ \lim_{n\to \infty} \frac{\dim H_{1}(B_n,K)}{[G \colon B_n]}= -\chi(G).\]
\end{lemma}

We state other results from \cite{Desi2} that will be important for us, as we will show that residually Droms RAAGs that are of type $FP_m$ for some $m \geq 2$ satisfy virtually the assumptions of Theorem \ref{residual}.

\begin{theorem}\cite[Theorem F]{Desi2} \label{residual}
Let $G \subseteq G_{1}\times \cdots \times G_{k}$ be a subdirect product of residually finite groups of type $F$, each of which contains a normal free subgroup $F_{i} < G_{i}$ such that $G_{i} \slash F_{i}$ is torsion-free and nilpotent. Assume $F_{i} \subseteq G \cap G_i$. Let $m<k$ be an integer, let $K$ be a field, and suppose that each $G_{i}$ is $K$-slow above dimension $1$.

If the projection of $G$ to each $m$-tuple of factors $G_{j_{1}} \times \cdots \times G_{j_{m}}$ is of finite index, then there exists an exhausting normal chain $(B_n)$ in $G$ so that for $0\leq j \leq m$, \[ \lim_{n\to \infty} \frac{\dim H_{j}(B_n,K)}{[G \colon B_n]}=0.\]
\end{theorem}

\begin{theorem} \cite[Theorem G]{Desi2} \label{residual2} Every finitely presented residually free group $G$ that is not a limit group over a free group admits an exhausting normal chain $(B_n)$  with respect to which the rank gradient \[ RG(G,(B_n))= \lim_{n\to \infty} \frac{d(B_n)}{[G \colon B_n]}=0.\]
Furthermore, if $G$ is of type $FP_3$ but it is not commesurable with a product of two limit groups over free groups, $(B_n)$ can be chosen so that the deficiency gradient $DG(G,(B_n))=0.$
\end{theorem}

\subsection{Preliminaries on groups that are $K$-slow above dimension $1$ and slow above dimension $1$}
\label{subsection slow}
Let us start recalling the definitions from \cite{Desi2} about slowness and $K$-slowness.

\begin{definition}
Let $G$ be a group. A sequence of non-negative integers $(r_{j})_{j\geq 0 }$ is a \emph{volume vector} for $G$ if there is a classifying space $K(G,1)$ that, for all $j\in \mathbb{N}$, has exactly $r_{j}$  $j$-cells.
\end{definition}

\begin{definition}
A group $G$ of homotopical type $F$ is \emph{slow above dimension $1$} if it is residually finite and for every exhausting normal chain $(B_n)$, there exist volume vectors $(r_{j}(B_{n}))_{j}$ for $B_n$ with finitely many non-zero entries, so that \[ \lim_{n\to \infty} \frac{r_{j}(B_{n})}{[G \colon B_{n}]} =0,\] for all $j\geq 2$.
$G$ is \emph{slow} if it satisfies the additional requirement that the limit exists and is zero for $j=1$ as well.
\end{definition}

\begin{example}\cite[Examples 4.4]{Desi2}\label{Example Desi}
Finitely generated torsion-free nilpotent groups are slow. The trivial group is slow. Free groups are slow above dimension $1$. Surface groups are slow above dimension $1$.
\end{example}

\begin{proposition}\label{Proposition 1}\cite[Proposition 4.5]{Desi2}
If a residually finite group $G$ is the fundamental group of a finite graph of groups where all of the edge groups are slow and all of the vertex groups are slow above dimension $1$, then $G$ is slow above dimension $1$.
\end{proposition}

\begin{definition}
Let $K$ be a field and let $G$ be a residually finite group. $G$ is \emph{$K$-slow above dimension 1} if for every exhausting normal chain $(B_n)$, we have
\[ \lim_{n\to \infty} \frac{\dim_{K}H_{j}(B_{n},K)}{[G \colon B_{n}]}=0,\] for all $j\geq 2$.

$G$ is \emph{$K$-slow} if it satisfies the additional requirement that the limit exists and is zero for $j=1$ as well.
\end{definition}

It follows directly by the definitions that if a group $G$ is slow above dimension 1 (respectively, slow), then it is $K$-slow above dimension 1 (respectively, $K$-slow).

\begin{proposition}\label{Proposition 2}\cite[Proposition 5.3]{Desi2}
Let $K$ be a field. If a residually finite group $G$ is the fundamental group of a finite graph of groups where all of the edge groups are $K$-slow and all of the vertex groups are $K$-slow above dimension $1$, then $G$ is $K$-slow above dimension $1$.
\end{proposition}

\subsection{Limit groups over coherent RAAGs are slow above dimension $1$}
\label{coherent}

 \begin{lemma} \cite[Corollary 9.7]{Montse} Limit groups over coherent RAAGs are $CAT(0)$.   \end{lemma}

In particular, since limit groups over coherent RAAGs are torsion-free, they are of type $F$.

\begin{lemma} \label{coherentRAAGs} Coherent RAAGs are slow above dimension $1$. In particular, Droms RAAGs are  slow above dimension $1$.
\end{lemma}

\begin{proof} 
In \cite{Droms} Droms proved that if $GX$ is a coherent RAAG, then $GX$ splits as a finite graph of groups where all the vertex groups are free abelian. Thus, by Proposition \ref{Proposition 1}, coherents RAAGs are  slow above dimension $1$. 
\end{proof}

The goal of this section is to show that limit groups over coherent RAAGs are slow above dimension $1$. For that, we will use the work on \cite{Montse} and \cite{Montse2} about limit groups over coherent RAAGs.

Let us recall some definitions that are needed to understand the results below. If $GX$ is a RAAG, the elements of $X$ are called the \emph{canonical generators} of $GX$. A \emph{non-exceptional surface} is a surface which is not a non-orientable surface of genus $1$, $2$ or $3$.  We can use the following proposition to define graph towers over coherent RAAGs.

\begin{proposition}\cite[Lemma 7.3]{Montse}\label{Proposition graph tower}
Let $G$ be a coherent right-angled Artin group. A graph tower over $G$ of height $0$ is a coherent RAAG $H$ which is obtained from $G$ by extending centralizers of canonical generators of $G$.

A graph tower over $G$ of height $\geq 1$ can be obtained as a free product with amalgamation, where the edge group is a free abelian group, one of the vertex groups is a graph tower over $G$ of lower height and the other vertex group is either free abelian, or the direct product of a free abelian group and the fundamental group of a non-exceptional surface.
\end{proposition}

\begin{theorem}\cite[Theorem 8.1]{Montse2}\label{Proposition graph tower 2}
Let $\Gamma$ be a limit group over a RAAG $G$. Then, $\Gamma$ is a subgroup of a graph tower over $G$.
\end{theorem}

\begin{lemma}\label{Lemma CW}\cite{Wall}
Let $1\rightarrow C \rightarrow D \rightarrow E \rightarrow 1$ be a short exact sequence of groups of type $F$. Suppose that there are classifying spaces $K(C,1)$ and $K(E,1)$ with $\alpha_{t}(C)$ and $\alpha_{t}(E)$ $t$-cells, respectively. Then, there is a $K(D,1)$ complex with $\alpha_{i}(D)$ $i$-cells such that
\[ \alpha_{i}(D)= {\mathlarger{\sum}}_{0\leq t \leq i} \alpha_{t}(C)\alpha_{i-t}(E).\]
\end{lemma}

We now prove some results that will be used in order to show that limit groups over coherent RAAGs are slow above dimension $1$.

\begin{lemma} \label{obvious}  Suppose that $G$ is a group of homotopical type $F$, $H$ is a normal subgroup of finite index in $G$ and $(r_{j}(G))_{j}$ is a volume vector for $G$. Then,  $([G \colon H]r_{j}(G))_{j}$ is a volume vector for $H$.
\end{lemma}

\begin{proof}
Let $Y$ be a classifying space $K(G,1)$ that, for all $j\in \mathbb{N} \cup \{ 0 \} $, has exactly $r_{j}(G)$  $j$-cells. Then, if we denote by $\widetilde{Y}$ the universal cover of $Y$, $\widetilde{Y}$ is contractible and $Y= \widetilde{Y} \slash G$. Therefore, $\widetilde{Y} \slash H$ is a classifying space for $H$ with exactly $[G \colon H]r_{j}(G)$ $j$-cells.
 \end{proof}

\begin{lemma} \label{slow-123} Let $G$ be a group of homotopical type $F$ and $H$ a  residually finite group where there is a short exact sequence
\[
1 \to \mathbb{Z}^n \to H \to G \to 1 .
\]
a) If $n \geq 1$, then $H$ is slow.\\[5pt]
b) If $n \geq 0$ and $G$ is slow above dimension $1$, then $H$ is slow above dimension $1$.
\end{lemma}

\begin{proof}
Let us denote $\mathbb{Z}^n$ by $A$.

a) Let $(B_i)$ be a chain of finite index normal subgroups in $H$ with $\bigcap B_i=1$. We need to show that for each $i$ there is a $K(B_i,1)$ complex with $r_j(B_i)$ $j$-cells such that for each $j\geq 0$,
\[ \lim_{i\to \infty} \frac{ r_{j}(B_{i})}{[ H  \colon B_{i}]} =0.\] 
For each $i$, the short exact sequence from the statement induces a short exact sequence
\[ 
1 \to A \cap B_{i} \to B_{i} \to p(B_{i}) \to 1.
\]
Let us show that \[[H \colon B_i]= [A \colon A \cap B_i][G \colon p(B_i)].\] Indeed, note that $[H \colon B_i]= [H \colon AB_i][AB_i \colon B_i]$. Firstly, $[AB_i \colon B_i]$ equals $[A \colon A \cap B_i]$. Secondly, $[H \colon A B_i]$ equals $[H \slash A \colon AB_i \slash A]$, and $H \slash A \simeq G $ and $ AB_i \slash A \simeq p(B_i).$
Therefore, $[H \colon AB_i] =[G \colon p(B_i)].$

Let  $\alpha_{j}(G)$ be the number of $j$-cells in a fixed  $K(G,1)$ complex. By  Lemma \ref{obvious}, there is a  $K(p(B_{i}),1)$ complex with  $\alpha_{j}(p(B_{i}))$  $j$-cells such that 
\[\alpha_{j}(p(B_{i})) = [ G \colon p(B_{i})]\alpha_{j}(G).\]
Since $A\cap B_{i}$ has finite index in $A$, there is a $K(A\cap B_{i},1)$ complex with ${n \choose j}$ $j$-cells.
By Lemma \ref{Lemma CW}, there is a $K(B_{i},1)$ complex with $\alpha_{j}(B_{i})$ $j$-cells such that
\[ \alpha_{j}(B_{i})=  {\mathlarger{\sum}}_{0\leq a \leq j} {n \choose {j-a}} \alpha_{a}(p(B_i))= {\mathlarger{\sum}}_{0\leq a \leq j} {n \choose {j-a}} [G \colon p(B_i)] \alpha_{a}(G), \]
and we set $r_j(B_i) = \alpha_j(B_i)$. Then,
 \[ \lim_{i\to \infty} \frac{ r_{j}(B_{i})}{[ H  \colon B_{i}]} = \lim_{i\to \infty} \frac{ [ G \colon p(B_{i})]{\mathlarger{\sum}}_{0\leq a \leq j} {n \choose j-a}\alpha_{a}(G)}
 {[G \colon p(B_{i})][A \colon B_{i} \cap A]} =\]\[{\mathlarger{\sum}}_{0\leq a \leq j} {n \choose j-a}\alpha_{a}(G) \lim_{i\to \infty} \frac{1}{[A \colon B_{i} \cap A]} = 0.
 \]
b) If $n\geq 1$, then we apply a). If $n=0$, by assumption $G$ is slow above dimension $1$.

\end{proof}

\begin{theorem} \label{slow-coherent-1} Limit groups over coherent RAAGs are slow above dimension $1$.
\end{theorem}

\begin{proof}
Let $G$ be a coherent RAAG and let $\Gamma$ be a limit group over $G$. Then, by Theorem \ref{Proposition graph tower 2}, $\Gamma$ is a subgroup of a graph tower over $G$, say $L$.
Let us prove by induction on the height of $L$ that $\Gamma$ is slow above dimension $1$ (see Proposition \ref{Proposition graph tower}).

If $L$ has height $0$, $L$ is a coherent RAAG. Therefore, $L$ is the fundamental group of a graph of groups where the vertex groups are free abelian. Thus, $\Gamma$ also admits a decomposition as a graph of groups where the vertex groups are free abelian, so by Proposition \ref{Proposition 1} we get that $\Gamma$ is slow above dimension $1$.

Now suppose that the height of $L$ is greater than $0$. Then, $L$ is a free product with amalgamation, where the edge group is a free abelian group, one of the vertex groups is a graph tower over $G$ of lower height and the other vertex group is either free abelian, or the direct product of a free abelian group and the fundamental group of a non-exceptional surface. Then, $\Gamma$ admits a decomposition as a graph of groups where the edge groups are free abelian (including the possibility that some are trivial), and the vertex groups are either subgroups of graph towers over $G$ of lower height (and by induction, those vertex groups are slow above dimension $1$) or free abelian groups or subgroups of the direct product of a free abelian group and the fundamental group of a non-exceptional surface. It suffices to show that finitely generated subgroups of the direct product of a free abelian group and the fundamental group of a non-exceptional surface are slow above dimension $1$. Then, by Proposition \ref{Proposition 1} we obtain that $\Gamma$ is slow above dimension $1$. 

If $H$ is a finitely generated subgroup of $\mathbb{Z}^m \times G_0$ where $m\in \mathbb{N} \cup \{0\}$ and $G_0$ is the fundamental group of a non-exceptional surface, then there is a short exact sequence
\[ 
1 \to A \to H \to  N \to 1,
\]
where $A$ is a free abelian group and $N$ is a finitely generated subgroup of $G_0$. In particular, $N$ is either a surface group or a free group, so $N$ is slow above dimension $1$. Then, by Lemma \ref{slow-123}, $H$ is slow above dimension $1$.
\end{proof}

\begin{remark}   Observe that by Theorem \ref{slow-coherent-1} coherent RAAGs  are slow above dimension 1. It is not true that every RAAG is  slow above dimension 1, as this would imply that it is $K$-slow above  dimension 1 for every field $K$. Indeed, by
\cite[Theorem 1]{Modp}, if $G$  is a right-angled Artin group  with underlying flag complex $L$, for any exhausting normal chain $(B_{n})$ in $G$,
\[ \lim_{n\to \infty} \frac{\dim_{K} H_{j}(B_{n}, K)}{[G \colon B_{n}]}= \bar{b}_{j-1}(L, K),\]
where $\bar{b}_{j-1}(L,K)$ denotes the reduced Betti number of $L$ with coefficients in $K$.
\end{remark}

\begin{proof}[Proof of Theorem \ref{Theorem C}]
Parts (1) and (2) follow from Theorem \ref{Theorem D2} and Theorem \ref{slow-coherent-1}. Part (3) follows from Theorem \ref{slow-coherent-1}.    
\end{proof}

\begin{proof}[Proof of Theorem \ref{Theorem D}]
It follows from Lemma \ref{homology-slow-limit-free}  and the fact that Theorem \ref{slow-coherent-1} implies that every limit group over a coherent RAAG is $K$-slow above dimension $1$.    
\end{proof}

\begin{proof}[Proof of Theorem \ref{Theorem E}]
The proof is an adaptation of the proof from \cite{Desi2} to the setting of residually Droms RAAGs. Thus, we just give a sketch of the proof.

From Theorem \ref{Theorem 10.1} we get that $G$ is a full subdirect product of limit groups over Droms RAAGs $G_0 \times G_{1}\times \cdots \times G_{r}$ such that $G_0$ is abelian (possibly trivial), $G \cap G_0$ has finite index in $G_0$ and $G_i$ has trivial center for $i\in \{1, 2,\dots,r\}$. 
By Proposition~\ref{free-by-(torsion-free nilpotent)}, $G_{i}$ is free-by-(torsion-free nilpotent).

 Suppose that $G_0$ is trivial. Then, part (1) from Theorem \ref{Theorem E} follows from Theorem \ref{residual} and Theorem \ref{FPs}.  Indeed, by Proposition \ref{free-by-(torsion-free nilpotent)}, each $G_i$  has a free normal subgroup $N_i$ such that $G_i \slash N_i$ is torsion-free nilpotent. By \cite[Theorem 6.2]{Jone}, $G_i\slash (G \cap G_i)$ is virtually nilpotent, hence $F_i = G \cap N_i$ is a free normal subgroup of $G_i$ such that $G_i/ F_i$ is virtually nilpotent. Hence
 there is a subgroup of finite index $\widetilde{G}_i$ in $G_i$  that contains $F_i$  such that  $\widetilde{G}_i \slash F_i$ is torsion-free nilpotent. Then we can apply Theorem \ref{residual} for $\widetilde{G} \colon = G \cap ( \widetilde{G}_1 \times \cdots \times \widetilde{G}_r) <  \widetilde{G}_1 \times \cdots \times \widetilde{G}_r$. This guarantees an exhausting normal chain $(B_n)$ of $\widetilde{G}$ that in general might not be a normal chain in $G$  but it is still an exhausting chain in $G$ such that for $0\leq j \leq m$, $ \lim_{n\to \infty} \frac{\dim H_{j}(B_n,K)}{[G \colon B_n]}=0$.

If $G_0$ is non-trivial, $G \cap G_0$ is non-trivial, free abelian and central, so part (1) from Theorem \ref{Theorem E} follows from \cite[Lemma 7.2]{Desi2}.

It remains to consider part (2) from Theorem \ref{Theorem E}. From Theorem \ref{Theorem 3.1}, $G$ has a subgroup of finite index $H= H_1 \times \cdots \times H_{ r_0}$ where each $H_i$ is a limit group over  Droms RAAGs. Following the proof of \cite[Theorem E]{Desi2}, we can choose $ ( B_i)$ an exhausting normal chain in $H$ and such that $B_i =  ( B_i  \cap  H_1 ) \times \cdots \times (B_i \cap H_{ r_0})$. Then, as in \cite{Desi2}, we can deduce that for $j \geq 1$,
\[ \lim_{n  \to \infty } \frac{\dim H_{j}(B_n, K)}{[G \colon B_n]}=(-1)^{ r_0} \delta_{j,{ r_0}} \chi(G),\]
where $\delta_{j,{ r_0}}$ is the Kronecker symbol.

A limit group over a Droms RAAG does not contain a direct product of two or more non-abelian free groups, since limit groups over Droms RAAGs are coherent (see \cite[Corollary 7.8]{Montse}) and the direct product of two non-abelian free groups is not coherent. In addition, every limit group over a Droms RAAG that has trivial center contains a non-abelian free group (see, for instance, \cite[Property 2.10]{Jone}), so ${ r_0}= \rho$ unless one or more $H_i$ has non-trivial center. If some $H_i$ has non-trivial center, $\chi(H_i)= 0$ and so $\chi(H) = 0$ and $ \chi(G)=0$.    
\end{proof}

\begin{proof}[Proof of Theorem \ref{Theorem F}]
In Theorem \ref{residual2}
 M. Bridson and D. Kochloukova proved  a low dimensional homotopical version of Theorem \ref{Theorem E} for limit groups over free groups. The proof of  Theorem \ref{residual2}  relies on some general results about residually finite groups, see \cite[Lemma 8.1, Lemma 8.2]{Desi2}, and some specific results for limit groups over free groups. The properties of limit groups over free groups that are used in the proof of  Theorem \ref{residual2}  still hold for the class of limit groups over Droms RAAGs and correspond to Theorem \ref{FPs} and Theorem \ref{Theorem 10.1}, so the same proof can be used in order to show Theorem \ref{Theorem F}.    
\end{proof}

\end{document}